\documentclass[preprint,review,12pt]{amsart}

\usepackage{amsfonts,amssymb,amsmath,amsthm,amsfonts}

\newtheorem{theorem}{Theorem}[section]
\newtheorem{lemma}[theorem]{Lemma}

\theoremstyle{definition}

\theoremstyle{remark}
\newtheorem{remark}[theorem]{Remark}

\numberwithin{equation}{section}

\DeclareMathOperator*{\esup}{ess\,sup}

\begin{document}

\title[On compactness of Laplace and Stieltjes transformations]{On compactness of Laplace and Stieltjes type transformations in Lebesgue spaces}

\author{Elena P. Ushakova}

\curraddr{Department of Mathematics, University of York, York, YO10 5DD, UK.}\email{elena.ushakova@york.ac.uk}

\address{Computing Centre of the Far-Eastern Branch of the Russian Academy of Sciences, Khabarovsk, 680000, RUSSIA.}
\email{elenau@inbox.ru}

\begin{abstract}
We obtain criteria for integral transformations of Laplace and Stieltjes type to be compact on Lebesgue spaces of real functions on the semiaxis.
\end{abstract}



\maketitle

\section{INTRODUCTION}
\label{I}

Let $L^r(I)$ denote the Lebesgue space of all measurable functions $f(x)$ integrable to a power $0<r<\infty$ on an interval $I\subseteq[0,\infty)=:
\mathbb{R}^+,$ that is
$$L^r(I)=\Bigl\{f\colon \|f\|_{r,I}:=\Bigl(\int_I|f(x)|^r\mathrm{d}x\Bigr)^{1/r}<\infty\Bigr\}.$$
If $I=\mathbb{R}^+$ we write $L^r:=L^r(\mathbb{R}^+),$ and
$\|f\|_r$ means $\|f\|_{r,\mathbb{R^+}}.$
For $r=\infty$ we denote $$L^\infty(I)=\Bigl\{f\colon \|f\|_{\infty,I}:=\esup_{t\in I}|f(t)|<\infty\Bigr\}.$$

Take $\lambda>0,$ $p \ge 1,$ $q>0$ and put $p':=p/(p-1).$ Assume $v\in L^{p'}_{loc}(\mathbb{R}^+)$ and $w\in L^q_{loc}(\mathbb{R}^+)$ are weight functions. In this article we study compactness properties of two particular cases of an integral transformation $T:L^p\to L^q$ of the form \begin{equation}\label{0} Tf(x):=w(x)\int_{\mathbb{R}^+}k_T(x,y)f(y)v(y)\mathrm{d}y, \hspace{1cm}x\in\mathbb{R}^+,\end{equation}
with a non-negative kernel $k_T(x,y)$ decreasing in variable $y$. We take as $T$ the Laplace integral operator
\begin{equation}\label{L}
\mathcal{L}f(x):=\int_{\mathbb{R}^+}\mathrm{e}^{-xy^\lambda}f(y)v(y)\mathrm{d}y,\hspace{1cm}
x\in\mathbb{R}^+,\end{equation} with the outer weight function $w(x)=1$ and $k_L(x,y):=\mathrm{e}^{-xy^\lambda},$
and a Stieltjes type transformation of the form
\begin{equation}\label{S}
Sf(x):=w(x)\int_{\mathbb{R}^+}\frac{f(y)v(y)\mathrm{d}y}{x^\lambda+y^\lambda},\hspace{1cm}
x\in\mathbb{R}^+,\end{equation} with $k_S(x,y):=(x^\lambda+y^\lambda)^{-1}.$ These operators are related to each other by \eqref{rep}.

With an appropriate choice of $\lambda,$ $v$ and $w$ transformations $\mathcal{L}$ and $S$ become special cases of conventional convolution transformation $F(x)=\int_{-\infty}^\infty f(t)G(x-t)\mathrm{d}t,$ $-\infty<x<\infty$ \cite[Ch.8, \S\S~8.5, 8.6]{Z}. The Stieltjes type operator \eqref{S} has also connections with Hilbert's double series theorem (see \cite{A} for details). Some interesting properties and applications of the Laplace type transform \eqref{L} to differential equations are indicated in \cite[Ch. 5]{KKOP}.

In this work we find explicit necessary and sufficient conditions for $L^p-L^q$--compactness of $\mathcal{L}$ and $S$ expressed in terms of kernels $k_L,k_S,$ weight functions $v,w$ and properties of the Lebesgue spaces. The results may be useful for study of characteristic values of the transformations. All cases of summation parameters $p\ge 1$ and $q>0$ are considered. If $0<p<1$ then $T:L^p\to L^q$ is compact in trivial case only (see \cite[Theorem 2]{PS}). Note that $L^2-L^2$ compactness of \eqref{L} and \eqref{S} was studied in \cite{DAN10, SMZ10}. We generalize these results for all positive $p$ and $q.$

Our main method is well-known and consists in splitting an initial operator into a sum of a compact operator and operators with small norms (see e.g. \cite{EGP}, \cite{LomStep}, \cite{Osmz}).

The article is organized as follows. Section 3 is devoted to the compactness of the Laplace transformation \eqref{L}. Criteria for the compactness of the Stieltjes operator \eqref{S} appear in Section 4. Note that the case of negative $\lambda$ in $S$ ensues from the results for positive $\lambda$ by simple modification of the weight functions $v$ and $w.$ In Section 5 we discuss cases $p=\infty$ and $q=\infty.$ Some auxiliary results are collected in Section 2.

Throughout the article we assume $v,w$ to be non-negative.
Products of the form $0\cdot\infty$ are supposed to be equal to $0.$ An equivalence $A\approx B$ means either $A=c_0B$ or $c_1A\le B\le c_2A,$ where $c_i,$ $i=0,1,2$ are constants depending on $\lambda,p,q$ only. The symbol $\mathbb{Z}$ denotes integers, $\chi_E$ stands for a characteristic function of a subset $E\subset\mathbb{R}^+.$ In addition, we use $=:$ and $:=$ for marking new quantities.

\section{PRELIMINARIES}

In this part we set auxiliary results concerning boundedness of the transformations and adduce some known results on compactness of integral operators we shall need later on in our proofs.

\subsection{Boundedness}
We start from a statement for an integral operator $T$ from $L^p$ to $L^q,$ when $p=1.$

\begin{theorem}\label{p1}\cite[Theorem 3.2]{SinStep} Suppose $0<q<\infty$ and a non-negative kernel $k_T(x,y)$ of the operator \eqref{0} is non-increasing in $y$ for each $x.$ For all $f\ge 0$ the best constant $C$ in the inequality
\begin{equation}\label{-}\biggl(\int_{\mathbb{R}^+} \biggl(\int_{\mathbb{R}^+}k_T(x,y)f(y)v(y)\mathrm{d}y\biggr)^q w^q(x)\mathrm{d}x\biggr)^{1/q}\le C\int_{\mathbb{R}^+} f(y)\mathrm{d}y\end{equation}
is unchanged when $v(y)$ is replaced by ${\bar v}_0(y),$ where $${\bar v}_{c_1}(t):=\esup_{c_1<x<t}v(x).$$ \end{theorem}

Consider the Laplace operator $\mathcal{L}.$ Denote $r:=\frac{p\,q}{p-q},$ $q':=\frac{q}{q-1},$  $$\alpha^q:=\min\{2,2^{q-1}\},\hspace{.5cm}\beta^q:=\begin{cases}\frac{2}{q-1}, & 1<q\le 2,\\ 2^{q-1}, & q>2,\end{cases}\hspace{5mm}V_{c_1}(t):=\int_{c_1}^t v^{p'},$$
$$F_{c_1}(y):=\int_{c_1}^y f(t)v(t)\mathrm{d}t,\ \ \mathcal{F}(x):=\int_{c_1}^{c_2}\mathrm{e}^{-xy^ \lambda}f(y)v(y)\mathrm{d}y,$$ $$ A_{\mathcal{L}, \langle c_1,c_2\rangle}(t):=\bigl(t^{-\lambda}-c_2^{-\lambda}\bigr)^{1/q} \bigl[V_{c_1}(t)\bigr]^{1/p'},\ \ A_{\mathcal{L},\langle c_1,c_2\rangle}:=\sup_{c_1<t<c_2}A_{\mathcal{L}, \langle c_1,c_2\rangle}(t),$$ $$A_\mathcal{L}:=A_{\mathcal{L},\mathbb{R}^+},\ \ B_{\mathcal{L},\langle c_1,c_2\rangle}:=\left(\int_{c_1}^{c_2} \Bigl[t^{-\lambda}-c_2^{-\lambda}\Bigr]^{r/q} \bigl[V_{c_1}(t)\bigr]^{r/q'}v^{p'}(t)\mathrm{d}t\right)^{1/r},$$
$$B_\mathcal{L}:=B_{\mathcal{L},\mathbb{R}^+}= \biggl(\int_{\mathbb{R}^+}B_\mathcal{L}(t)\mathrm{d}t\biggr)^{1/r}, \ \  B_\mathcal{L}(t):=t^{-\lambda r/q}\bigl[V_{0}(t)\bigr]^{r/q'}v^{p'}(t),$$
$$B_{q}(t):=t^{-\lambda/q}v(t),\ \ {\bar B}_{q}(t):=t^{-\lambda/q}{\bar v}_0(t),\ \ B_p:=\biggl(\int_{\mathbb{R}^+}y^{-\lambda p'}v^{p'}(y)\mathrm{d}y\biggr)^{1/p'},$$
$$B_{q',\langle c_1,c_2\rangle}:=\biggl(\int_{c_1}^{c_2} [t^{-\lambda}-c_2^{-\lambda}]^{q/(1-q)}t^{-\lambda-1}{\bar v}_{c_1}(t)^{q/(1-q)}\mathrm{d}t\biggr)^{(1-q)/q},$$
$$ B_{q'}=B_{{q'},\mathbb{R}^+}=\left(\int_{\mathbb{R}^+}B_{q'}(t)\mathrm{d}t\right)^{(1-q)/q}, \ \ B_{q'}(t):=t^{-\lambda/(1-q)-1}{\bar v}_0(t)^{q/(1-q)},$$
$$D_{\langle c_1,c_2\rangle}:=c_2^{-\lambda/q}
\bigl[V_{c_1}(c_2)\bigr]^{1/p'}.$$

Various conditions were found for the boundedness of the Laplace transformation \eqref{L} in Lebesgue spaces  (see e.g. \cite{AH83}, \cite{B}). Convenient for our purposes $L^p-L^q$ criterion for $\mathcal{L}$ was obtained in  \cite[Theorem 1]{SUsmz} (see also \cite[Theorem 1]{PSUdan}).

\begin{theorem}\label{T0}\cite{PSUdan, SUsmz} The following estimates are true for the norm of $\mathcal{L}:$\\
{\rm(i)} Denote $\alpha_1:=\alpha q^{-2/q}$ and $\beta_1:=\beta (q')^{1/p'}.$ If $1<p\le q<\infty$ then  \begin{equation*}\alpha_1 A_\mathcal{L}\le
\|\mathcal{L}\|_{L^p\to L^q}\le\beta_1 A_\mathcal{L}.\end{equation*}
{\rm(ii)} Let $1\le q<p<\infty.$ If $q=1$ then $\|\mathcal{L}\|_{L^p\to L^1}=B_p.$ If $q>1$ then
\begin{equation*}\alpha_2 B_\mathcal{L}\le \|\mathcal{L}\|_{L^p\to L^q}\le \beta_2 B_\mathcal{L}\end{equation*} with $\alpha_2:=\alpha(p'q/r)^{1/q'}q^{-1/q},$ $\beta_2:=\beta (p')^{1/q'}.$ \\
{\rm(iii)} Put $\alpha_3:=q^{-1/q},$ $\beta_3:=p^{1/p}(p')^{1/q'}q^{-2/q}r^{1/r}.$ If $0<q<1<p<\infty$ then \begin{equation*}\alpha_3\,\|B_{q}\|_{p'}\le \|\mathcal{L}\|_{L^p\to L^q}\le \beta_3 B_\mathcal{L}.\end{equation*}
{\rm(iv)} Let $0<q\le 1=p.$ If $0<q<1$ then
\begin{equation*}\alpha_4\ \esup_{t\in\mathbb{R}^+}B_{q}(t) \le \|\mathcal{L}\|_{L^1\to L^q}\le  \beta_4B_{q'},\end{equation*}
where $\alpha_4:=q^{-1/q}$ and $\beta_4:=\lambda^{(1-q)/q}q^{-2/q} (1-q)^{-(1-q)/q}.$ If $q=1$ then
\begin{eqnarray} \label{200}2^{-\lambda}\esup_{t\in\mathbb{R}^+}t^{-\lambda}{\bar v}_0(t)\le 2^{-\lambda}\sup_{t\in\mathbb{R}^+}t^{-\lambda}{\bar v}_0(t)\le \|\mathcal{L}\|_{L^1\to L^1}\nonumber\\\le \esup_{t\in\mathbb{R}^+}t^{-\lambda}{\bar v}_0(t)\le\sup_{t\in\mathbb{R}^+}t^{-\lambda}{\bar v}_0(t).\end{eqnarray}
\end{theorem}

\begin{remark} Integration by parts gives the equality \begin{equation*}B_\mathcal{L}^r=(\lambda p'/q)\int_{\mathbb{R}^+}\bigl[V_0(t)\bigr]^{r/p'}t^{-\lambda r/q-1}\mathrm{d}t, \end{equation*} which is true either $v\in L^{p'}[0,t],$ $t>0,$ or $q>1$ (see \cite[Remark, p. 8]{SinStep}).
\end{remark}
\begin{remark} $B_q(t)$ in (iv) may be replaced by ${\bar B}_q(t)$ (see Theorem \ref{p1}).
\end{remark}
\begin{remark}\label{sigma} The lower estimates in \eqref{200} can be proved by applying
a function $f_t(y)=t^{-1}\chi_{(t,2t)}(y),$ $t>0$ into \eqref{-} with $v$ replaced by ${\bar v}_0.$
\end{remark}

%\smallskip
The results (i) and (ii) of Theorem 2.1 rest on \cite[Lemma 1]{SUsmz}. The following statement is its modification for the Laplace operator of the form $f\to\mathcal{L}(f\chi_{\langle c_1,c_2\rangle}),$ where $0\le c_1<c_2\le\infty$ and $\langle\cdot,\cdot\rangle$ denotes any of intervals $(\cdot,\cdot)$, $[\cdot,\cdot]$, $[\cdot,\cdot)$ or $(\cdot,\cdot].$

\begin{lemma}\label{L0}
Let $1<q<\infty.$ Assume $f\ge 0$ and suppose  the following conditions are satisfied: $\int_{\mathbb{R}^+}\mathcal{F}^q(x)\mathrm{d}x<\infty,$  $\lim\limits_{t\to c_1}t^{-\lambda}\bigl[V_{c_1}(t)\bigr]^{q/p'}=0,$
$\lim\limits_{t\to c_2}t^{-\lambda}\bigl[V_{c_1}(t)\bigr]^{q/p'}<\infty.$ Then \begin{eqnarray*}
\alpha^q\lambda q^{-2}\int_{c_1}^{c_2}F_{c_1}^q(y)
y^{-\lambda-1}\mathrm{d}y+\alpha^qq^{-2}
F_{c_1}^q(c_2)c_2^{-\lambda}
\le\int_{\mathbb{R}^+}\mathcal{F}^q(x)\mathrm{d}x
\\\le (\beta^q\lambda/q)\int_{c_1}^{c_2}F^q_{c_1}(y)
y^{-\lambda-1}\mathrm{d}y+(\beta^q/q)
F^q_{c_1}(c_2)c_2^{-\lambda}.\end{eqnarray*}
\end{lemma}

Lemma \ref{L0} modifies Theorem \ref{T0} as follows.

\begin{theorem}\label{T1} {\rm(i)} Let $1<p\le q<\infty.$
The operator $\mathcal{L}$ is bounded from $L^p\langle c_1,c_2\rangle$ to $L^q$ if and only if $A_{\mathcal{L},\langle c_1,c_2\rangle}+D_{\langle c_1,c_2\rangle}<\infty.$  Moreover, \begin{equation*}\alpha_1\bigl[A_{\mathcal{L}, \langle c_1,c_2\rangle}+D_{\langle c_1,c_2\rangle}\bigr]\le \|\mathcal{L}\|_{L^p\langle c_1,c_2\rangle\to L^q}\le\beta_1\bigl[A_{\mathcal{L}, \langle c_1,c_2\rangle}+D_{\langle c_1,c_2\rangle}\bigr].\end{equation*}
{\rm(ii)} If $1<q<p<\infty$ then $\mathcal{L}$ is bounded iff $B_{\mathcal{L},\langle c_1,c_2\rangle}+D_{\langle c_1,c_2\rangle}<\infty,$ where
\begin{equation*}\alpha_2\bigl[B_{\mathcal{L},\langle c_1,c_2\rangle}+D_{\langle c_1,c_2\rangle}\bigr]\le \|\mathcal{L}\|_{L^p\langle c_1,c_2\rangle\to L^q}\le \beta_2\bigl[B_{\mathcal{L},\langle c_1,c_2\rangle}+D_{\langle c_1,c_2\rangle}\bigr].\end{equation*}
{\rm(iii)} Let $0<q<1<p<\infty.$ The Laplace operator $\mathcal{L}$ is bounded from $L^p\langle c_1,c_2\rangle$ to $L^q$ if $B_{\mathcal{L},\langle c_1,c_2\rangle}+D_{\langle c_1,c_2\rangle}<\infty.$ If $\mathcal{L}:L^p\langle c_1,c_2\rangle\to L^q$ is bounded then $\|B_q\|_{p',\langle c_1,c_2\rangle}<\infty.$ We also have
\begin{equation*}\alpha_3\,\|B_q\|_{p',\langle c_1,c_2\rangle}\le \|\mathcal{L}\|_{L^p\langle c_1,c_2\rangle\to L^q}\le \beta_3\bigl[B_{\mathcal{L},\langle c_1,c_2\rangle}+D_{\langle c_1,c_2\rangle}\bigr].\end{equation*}
{\rm(iv)} Let $0<q<1=p.$ If $\mathcal{L}$ is $L^1-L^q$--bounded then $\displaystyle\esup_{t\in\langle c_1,c_2\rangle}B_{q}(t)<\infty.$ The operator $\mathcal{L}$ is bounded from $L^1$ to $L^q$ if
$B_{q',\langle c_1,c_2\rangle}<\infty$ and $D_{\langle c_1,c_2\rangle}<\infty.$ Besides, \begin{equation*}
\alpha_4\,\esup_{t\in\langle c_1,c_2\rangle}B_{q}(t)\le \|\mathcal{L}\|_{L^1\langle c_1,c_2\rangle\to L^q}\le \beta_4 B_{q',\langle c_1,c_2\rangle}+q^{-2/q}D_{\langle c_1,c_2\rangle}.\end{equation*}
\end{theorem}
\begin{remark} The constants $\alpha_i,\beta_i,$ $i=1,2,3,4,$ in Theorem \ref{T1} are defined as in Theorem \ref{T0}. Moreover, if $\langle c_1,c_2\rangle=\mathbb{R}^+$ and \begin{equation}\label{comp}{\rm(i)}\,\lim_{t\to 0}A_\mathcal{L}(t)=0,{\rm(ii)}\lim_{t\to \infty}A_\mathcal{L}(t)=0\end{equation} we have $D_{\mathbb{R}^+}=0$ and Theorem \ref{T0} has become a case of Theorem \ref{T1}.\end{remark}

Now we start to consider the Stieltjes operator $S.$ Add some denotations \begin{equation*}\mathcal{V}_{t}(\infty):=\int_{t}^\infty \frac{v^{p'}(y)\mathrm{d}y}{y^{\lambda p'}} ,\ \ W_{c_1}(t):=\int_{c_1}^t w^{q},\  \ \mathcal{W}_{t}(c_2):=\int_{t}^{c_2}\frac{w^{q}(x)\mathrm{d}x}{x^{\lambda q}} .\end{equation*}
Boundedness criteria for $S$ in Lebesgue spaces were found in \cite{And, Er, Sin}.
The following Theorem \ref{T3} contains some of them.
\begin{theorem}\label{T3} {\rm(i)} \cite[Theorem 1]{And} The operator  $S$ is bounded from $L^p$ to $L^q$ for $1<p\le q<\infty$ if and only if $A_S:=\sup_{t\in\mathbb{R}^+}A_S(t)<\infty$ with
\begin{equation*}A_S(t):=t^\lambda
\biggl(\int_{\mathbb{R}^+}
\frac{w^q(x)\mathrm{d}x}{(x^\lambda+t^\lambda)^{q}}\biggr)^{1/q}
\biggl(\int_{\mathbb{R}^+}
\frac{v^{p'}(y)\mathrm{d}y}{(t^\lambda+y^\lambda)^{p'}}\biggr)^{1/p'}. \end{equation*} Moreover, $A_S\le\|S\|_{L^p\to L^q}
\le\gamma_S\cdot A_S$
with a constant $\gamma_S$ depending on $p,$ $q$ and $\lambda$ only. If $p=1\le q<\infty$ then $\|S\|_{L^p\to L^q}\approx A_{1,S},$ where
\begin{equation*}A_{1,S}:=\sup_{t\in\mathbb{R}^+}t^\lambda
\biggl(\int_{\mathbb{R}^+}
\frac{w^q(x)\mathrm{d}x}{(x^\lambda+t^\lambda)^{q}}\biggr)^{1/q}
\esup_{y\in\mathbb{R}^+}\frac{v(y)}{t^\lambda+y^\lambda}.\end{equation*}
{\rm(ii)}\cite[Theorem 2.1]{Sin} $S$ is bounded from $L^p$ to $L^q$ for $1<q<p<\infty$ if and only if $B_S<\infty$ with
\begin{equation*}B_S:=\biggl(\int_{\mathbb{R}^+}\left(\int_{\mathbb{R}^+}
\frac{w^q(x)\mathrm{d}x}{(x^\lambda+t^\lambda)^{q}}\right)^{r/q}
\biggl(\int_{\mathbb{R}^+}
\frac{v^{p'}(y)\mathrm{d}y}{(1+[y/t]^\lambda)^{p'}}\biggr)^{r/q'}
v^{p'}(t)\mathrm{d}t\biggr)^{1/r}.\end{equation*}
Besides, $\|S\|_{L^p\to L^q}\approx B_S$ with constants of equivalence  depending, possibly, on $p,$ $q$ and $\lambda.$ If $q=1$ then \begin{equation*}\|S\|_{L^p\to L^1}=\biggl(\int_{\mathbb{R}^+}\biggl(\int_{\mathbb{R}^+}
\frac{w(x)\mathrm{d}x}{x^\lambda+t^\lambda}\biggr)^{p'}
v^{p'}(t)\mathrm{d}t\biggr)^{1/p'}.\end{equation*}\end{theorem}

In view of the relation \begin{equation}\label{rel}\frac{1}{2}\left[Hf(x)+H^\ast f(x)\right]\le
Sf(x)\le Hf(x)+H^\ast f(x),\hspace{5mm}f\ge 0,\end{equation}
some properties of $S$ can be interpreted through the Hardy operator \begin{equation*}
Hf(x):=x^{-\lambda}w(x)\int_0^x f(y)v(y)\mathrm{d}y\end{equation*} and its dual transformation $H^\ast f(x):=w(x)\int_x^\infty f(y)y^{-\lambda}v(y)\mathrm{d}y.$
Alternative criteria for $L^p-L^q$--boundedness of $S$, in comparison with Theorem \ref{T3},  follow from results for Hardy integral operators and cover even the case $0<q<1$ (see theorems \ref{H0} and \ref{H0ast} for details).
\begin{theorem}\label{TH} {\rm(i)} If $1<p\le q<\infty$ then
the operator $S$ is bounded from $L^p$ to $L^q$ if and only if
$A_H+A_{H^\ast}<\infty,$ where $\|S\|_{L^p\to L^q}
\approx A_{H}+A_{H^\ast}$ and \begin{eqnarray}\label{H1} A_H:=\sup_{t\in\mathbb{R}^+}
A_H(t):=\sup_{t\in\mathbb{R}^+}\bigl[V_0(t)\bigr]^{1/p'}
\bigl[\mathcal{W}_t(\infty)\bigr]^{1/q},\\\label{H2} A_{H^\ast}:=\sup_{t\in\mathbb{R}^+}
A_{H^\ast}(t):=\sup_{t\in\mathbb{R}^+}
\bigl[\mathcal{V}_t(\infty)\bigr]^{1/p'}
\bigl[W_0(t)\bigr]^{1/q}.\end{eqnarray}
{\rm(ii)} If $0<q<1<p<\infty$ or $1<q<p<\infty$ then $S:L^p\to L^q$ if and only if $B_{H}+B_{H^\ast}<\infty,$ where $\|S\|_{L^p\to L^q}
\approx B_{H}+B_{H^\ast}$ and
\begin{eqnarray*}B_H:=\biggl(\int_{\mathbb{R}^+}\bigl[V_0(t)\bigr]^{r/p'}
\bigl[\mathcal{W}_t(\infty)\bigr]^{r/p}t^{-\lambda q}
w^{q}(t)\mathrm{d}t\biggr)^{1/r},\\ B_{H^\ast}:=\biggl(\int_{\mathbb{R}^+}\bigl[\mathcal{V}_t(\infty)\bigr]^{r/p'}
\bigl[W_0(t)\bigr]^{r/p}w^{q}(t)\mathrm{d}t\biggr)^{1/r}.\end{eqnarray*} Besides, if $q>1$ then \begin{eqnarray*}B_H=(q/p')^{1/r} \biggl(\int_{\mathbb{R}^+}\bigl[V_0(t)\bigr]^{r/q'}
\bigl[\mathcal{W}_t(\infty)\bigr]^{r/p}
v^{p'}(t)\mathrm{d}t\biggr)^{1/r},\\ B_{H^\ast}=(q/p')^{1/r}\biggl(\int_{\mathbb{R}^+} \bigl[\mathcal{V}_t(\infty)\bigr]^{r/q'}
\bigl[W_0(t)\bigr]^{r/q}t^{-\lambda p'}v^{p'}(t)\mathrm{d}t\biggr)^{1/r}.\end{eqnarray*}
{\rm(iii)} Let $0<q\le 1=p.$ If $0<q<1$ then the Stieltjes transformation $S$ is $L^1-L^q$--bounded if and only if $B_{1,H}+B_{1,H^\ast}<\infty,$ where $\|S\|_{L^1\to L^q}
\approx B_{1,H}+B_{1,H^\ast}$ and
\begin{eqnarray*}B_{1,H}:=\biggl(\int_{\mathbb{R}^+}{\bar v}_0(t)^{q/(1-q)}
\bigl[\mathcal{W}_t(\infty)\bigr]^{q/(1-q)}t^{-\lambda q}
w^{q}(t)\mathrm{d}t\biggr)^{(1-q)/q},\\
B_{1,H^\ast}:=\biggl(\int_{\mathbb{R}^+}\left[{t^{-\lambda}}{\bar v}_t(\infty)\right]^{q/(1-q)}
\bigl[W_0(t)\bigr]^{q/(1-q)}w^{q}(t)\mathrm{d}t\biggr)^{(1-q)/q}.\end{eqnarray*} If $q=1$ then \begin{equation*}\|S\|_{L^1\to L^1}\approx\sup_{t\in\mathbb{R}^+}{\bar v}_0(t)
\int_{t}^\infty x^{-\lambda}w(x)\mathrm{d}x
+\sup_{t\in\mathbb{R}^+}{\bar v}_t(\infty)t^{-\lambda}
\int_0^t w(x)\mathrm{d}x.\end{equation*}\end{theorem}

We conclude the paragraph by giving a general boundedness criterion for an integral operator $T$ defined by \eqref{0} and acting from $L^1$ to $L^q,$ when $1< q<\infty.$
\begin{theorem}\label{p=1}\cite[Ch. XI, \S 1.5, Theorem 4]{KA} Let $1< q<\infty.$ The operator $T$ with measurable on $\mathbb{R}^+\times\mathbb{R}^+$ kernel $k_T(x,y)\ge 0$ is bounded from $L^1$ to $L^q$ if and only if \begin{equation*}\|{\bf k}_T\|_{L^\infty[L^q]}:=\esup_{t\in\mathbb{R}^+}\|w(\cdot) k_T(\cdot,t)v(t)\|_{q}<\infty.\end{equation*} Besides, $\|T\|_{L^1\to L^q}=\|{\bf k}_T\|_{L^\infty[L^q]}.$
\end{theorem}

\subsection{Compactness}
Suppose $I\subset\mathbb{R}^+$ and $J\subset\mathbb{R}^+$ are intervals of finite Lebesgue measure, that is ${\rm mes}\,I:=\int_I \mathrm{d}x<\infty$ and ${\rm mes}\,J<\infty.$

Let $K$ be an integral operator from $L^p(I)$ to $L^q(J)$ of the form \begin{equation*}Kf(x)=\int_Ik(x,y)f(y)\mathrm{d}y,\hspace{1cm}x\in J.\end{equation*} Assume $K_0:L^p(I)\to L^q(J)$ is a positive operator \begin{equation*}K_0f(x)=\int_Ik_0(x,y)f(y)\mathrm{d}y,\hspace{1cm}x\in J,\end{equation*} such that $Kf\le K_0|f|.$ Then $K$ is called {\it a regular operator} (see \cite[\S\,2.2]{KZPS} or \cite[Definition 3.5]{RS}) with a majorant operator $K_0$. Note that the last inequality ensues from the expression $|k(x,y)|\le k_0(x,y)$ (see \cite[\S\,5.6]{KZPS}).

Below we adduce two remarkable results on compactness of linear regular integral operators by M.A. Krasnosel'skii, P.P. Zabreiko, E.I. Pustyl'nik and P.E. Sobolevskii \cite{KZPS}. The first theorem gives conditions for the compactness of integral operators from $L^p(I)$ to $L^q(J),$ when $0< q\le 1<p<\infty.$

\begin{theorem}\label{T2} \cite[p. 94, Theorem 5.8]{KZPS} Every linear regular integral operator $K:L^p(I)\to L^q(J)$  is compact, when $0< q\le 1<p<\infty.$
\end{theorem}

The second theorem is on compactness of linear regular integral operators with compact majorants. The statement as well as Theorem \ref{T2} is true for any $I\subseteq\mathbb{R}^+$ and $J\subseteq\mathbb{R}^+.$

\begin{theorem}\label{TD}  \cite[p. 97, Theorem 5.10]{KZPS} Take $|k(x,y)|\le k_0(x,y)$ and suppose $K_0$ is compact from $L^p(I)$ to $L^q(J)$ for $1<p\le\infty$ and $1\le q<\infty.$ Then the operator $K:L^p(I)\to L^q(J)$ is compact as well.
\end{theorem}

The next theorem is on relative compactness of a subset of $L^q(\mathbb{R}^+),$ $1\le q<\infty.$
\begin{theorem}\label{Ada}  \cite[Theorem 2.21]{Ad} Let $1\le q<\infty$ and $\Omega$ be a bounded subset of $L^q(\mathbb{R}^+).$ Then $\Omega$ is relatively compact if and only if for all $\varepsilon>0$ there exists $\delta>0$ and a subset $I\subset \mathbb{R}^+$ such that for every $g\in\Omega$ and every $h\in\mathbb{R}^+$ with $h<\delta$ \begin{equation*}\int_{\mathbb{R}^+}\left|g(x+h)-g(x)\right|^q\mathrm{d}x<\varepsilon^q, \hspace{10mm} \int_{\mathbb{R}^+\setminus I}\left|g(x)\right|^q\mathrm{d}x<\varepsilon^q.\end{equation*}
\end{theorem}

We finalize the section by giving a statement obtained in \cite[Lemma 4]{LomStep} for Banach function spaces $X$ and $Y$ with absolutely continuous (AC) norms. A space $X$ has AC-norm if for all $f\in X,$ $\|f\chi_{E_n}\|_X\to 0$ for every sequence of sets $\{E_n\}\subset\mathbb{R}^+$ such that $\chi_{E_n}(x)\to 0$ a.e. We shall apply this statement when $X=L^p,$ $1<p<\infty,$ and  $Y=L^q,$ $1\le q<\infty.$
\begin{lemma}\label{L2}\cite[Lemma 4]{LomStep}
Let $1<p<\infty,$ $1\le q<\infty$ and $T\colon L^p\to L^q$ be an integral operator of the form \eqref{0}. $T$ is compact if \begin{equation}\label{XY}
M_T:=\Bigl\|\bigl\|w(x)k(x,\cdot)v(\cdot)\bigr\| _{p'}\Bigr\|_{q}<\infty.
\end{equation}
\end{lemma}

\section{COMPACTNESS OF THE LAPLACE TRANSFORM}

\begin{theorem}\label{lapa}
{\rm(i)} If $1<p\le q<\infty$ then the operator $\mathcal{L}:L^p\to L^q$ is compact if and only if
$A_\mathcal{L}<\infty$ and the conditions \eqref{comp} are fulfilled. \\
{\rm(ii)} Let $1\le q<p<\infty.$ If $q>1$ then $\mathcal{L}:L^p\to L^q$ is compact if and only if $B_\mathcal{L}<\infty.$ If $q=1$ then $\mathcal{L}$ is compact if and only if $B_p<\infty.$\\
{\rm(iii)} Let $0<q<1<p<\infty.$ The operator $\mathcal{L}:L^p\to L^q$ is compact if $B_\mathcal{L}<\infty.$ If $\mathcal{L}$ is compact from $L^p$ to $L^q$ then $\|B_{q}\|_{p'}<\infty.$ \\
{\rm(iv)} Let $0<q<1=p.$ $\mathcal{L}$ is compact from $L^1$ to $L^q$ if $B_{q'}<\infty.$ If $\mathcal{L}$ is $L^1-L^q$--compact then $\esup_{t\in\mathbb{R}^+}B_q(t)<\infty.$\\
{\rm(v)} Let $1=p\le q<\infty.$ The operator $\mathcal{L}$ is $L^1-L^q$-compact if and only if
\begin{equation*}\esup_{t\in\mathbb{R}^+}{\bar B}_q(t)<\infty\ \textrm{and}\ \lim_{t\to 0}{\bar B}_q(t)=\lim_{t\to\infty}{\bar B}_q(t)=0.\end{equation*}
\end{theorem}

\begin{proof} (i)
{\it Sufficiency.} Suppose $A_\mathcal{L}<\infty$ and $\lim\limits_{t\to 0}A_\mathcal{L}(t)=\lim\limits_{t\to\infty}A_\mathcal{L}(t)=0.$ Put $0<a<b<\infty$ and denote \begin{equation*}\mathcal{L}_0f:=\mathcal{L}(f\chi_{(a,b)}),\
\ \mathcal{L}_1f:=\mathcal{L}(f\chi_{[0,a]}),\
\ \mathcal{L}_2f:=\mathcal{L}(f\chi_{[b,\infty)}).\end{equation*}
Obviously, \begin{equation}\label{15}\mathcal{L}f(x)=\sum_{i=0}^2\mathcal{L}_if(x).\end{equation}
Since $A_\mathcal{L}<\infty$ then $\mathcal{L}$ is bounded
from $L^p$ to $L^q$ by Theorem \ref{T0}(a). This yields $L^p-L^q$--boundedness of the operator $\mathcal{L}_0f,$ which is regular with a majorant operator $\mathcal{L}_0|f|$ bounded from $L^p$ to $L^q$ as well.

According to Lemma \ref{L2} the operator $\mathcal{L}_0\colon L^p\to L^q$ is compact if
\begin{equation}\label{DDL0}
M_{\mathcal{L}_0}:=\Bigl\|\bigl\|\chi_{(a,b)}(\cdot)k_{\mathcal{L}}(x,\cdot)v(\cdot)\bigr\| _{p'}\Bigr\|_{q}<\infty.\end{equation}
Since $0<a<b<\infty$ and $v^{p'}\in L_{loc}(\mathbb{R}^+)$ we have
\begin{eqnarray}\label{1}M_{\mathcal{L}_0}^q\le \frac{1}{a^\lambda q}
\bigl[V_a(b)\bigr]^{q/p'}<\infty
.\end{eqnarray}
Therefore, $\mathcal{L}_0$ is compact from $L^p$ to $L^q$ for any $0<a<b<\infty.$

Now consider the operators $\mathcal{L}_i,$ $i=1,2.$ By Theorem \ref{T1}(a) we have:
\begin{eqnarray}\label{10}\|\mathcal{L}_1\|_{L^p\to L^q}\le 2\beta_1\sup_{0\le t\le a}t^{-\lambda/q}\bigl[V_0(t)\bigr]^{1/p'},\\
\label{11}\|\mathcal{L}_2\|_{L^p\to L^q}\le\beta_1\sup_{b\le t<\infty}t^{-\lambda/q}\bigl[V_b(t)\bigr]^{1/p'}.
\end{eqnarray}
The conditions \eqref{comp}
yield
\begin{eqnarray*}\lim_{a\to 0}\sup_{0\le t\le a}t^{-\lambda/q}\bigl[V_0(t)\bigr]^{1/p'}=\lim_{t\to 0}A_\mathcal{L}(t)=0,\\
\lim_{b\to\infty}\sup_{b\le t<\infty}t^{-\lambda/q}\bigl[V_b(t)\bigr]^{1/p'}\le\lim_{t\to \infty}A_\mathcal{L}(t)=0.\end{eqnarray*} Together with \eqref{10} and \eqref{11} this gives:  \begin{equation}\label{13}
\lim_{a\to 0}\|\mathcal{L}_1\|_{L^p\to L^q}=0,\hspace{1cm}
\lim_{b\to\infty}\|\mathcal{L}_2\|_{L^p\to L^q}=0.\end{equation}
Therefore, \eqref{15} implies
\begin{eqnarray}\label{limit}
\|\mathcal{L}-\mathcal{L}_0\|_{L^p\to L^q}\le\|\mathcal{L}_1\|_{L^p\to L^q}+\|\mathcal{L}_2\|_{L^p\to L^q},
\end{eqnarray} and now the operator $\mathcal{L}:L^p\to L^q$ is compact as a limit of compact operators, when $a\to 0$ and $b\to\infty.$% by \eqref{13}.

{\it Necessity.} Suppose now $\mathcal{L}$ is compact from $L^p$ to $L^q.$ Then $\mathcal{L}$ is $L^p-L^q$--bounded and $A_\mathcal{L}<\infty$ by Theorem \ref{T0}(a).

To prove \eqref{comp} we assume $\{z_k\}_{k\in\mathbb{Z}}\subset\mathbb{R}^+$ is an arbitrary sequence. To establish the claim (i) in \eqref{comp} suppose $\lim_{k\to\infty}z_k=0$ and put
\begin{equation*}f_k(t)=\chi_{[0,z_k]}(t)[v(t)]^{p'-1} \bigl[V_0(z_k)\bigr]^{-1/p}.\end{equation*}
Since $\|f_k\|_p=1$ then
\begin{equation*}\biggl|\int_{\mathbb{R}^+}f_k(y)g(y)\mathrm{d}y\biggr|\le
\biggl(\int_0^{z_k}|g(y)|^{p'}\mathrm{d}y\biggr)^{1/p'}\to 0,\hspace{1cm}k\to\infty,\end{equation*}
for any $g\in L^{p'}.$ Therefore, the sequence $\{f_k\}_{k\in\mathbb{Z}}$ converges weakly to $0$ in $L^p.$ Compactness of $\mathcal{L}:L^p\to L^q$ yields strong convergence of $\{\mathcal{L}f_k\}_{k\in\mathbb{Z}}$ in $L^q,$ that is $\lim_{k\to\infty}\|\mathcal{L}f_k\|_q=0.$ Besides, we have \begin{equation*}
\int_0^\infty\biggl(\int_0^\infty \mathrm{e}^{-xy^\lambda}f_k(y)v(y)\mathrm{d}y\biggr)^q\mathrm{d}x
\ge\int_0^\infty \mathrm{e}^{-qxz_k^\lambda}\mathrm{d}x\bigl[V_0(z_k)\bigr]^{q/p'}
=\frac{A_\mathcal{L}^q(z_k)}{q}.\end{equation*} Hence, $\lim_{k\to\infty}A_\mathcal{L}(z_k)=0.$ Since $\{z_k\}_{k\in\mathbb{Z}}$ is arbitrary, then \eqref{comp}(i) is proved.

For the proof of \eqref{comp}(ii) we suppose $\lim_{k\to\infty}z_k=\infty$ and put
\begin{equation*}g_k(t)=\chi_{[0,z_k^{-\lambda}]}(t)z_k^{\lambda/q'}.\end{equation*} Since $\|g_k\|_{q'}=1$ we have \begin{equation*}\biggl|\int_0^\infty f(x)g_k(x)\mathrm{d}x\biggr|
\le \biggl(\int_0^{z_k^{-\lambda}}|f(x)|^{q}\mathrm{d}x\biggr)^{1/q}\to 0,\hspace{1cm}k\to \infty,\end{equation*} for any $f\in L^q,$ which means weak convergence of $\{g_k\}_{k\in\mathbb{Z}}$ in $L^{q'}.$ Compactness of $\mathcal{L}:L^p\to L^q,$ $1<p,q<\infty,$ implies $L^{q'}-L^{p'}$--compactness of the dual operator \begin{equation*}\mathcal{L}^\ast g(y):=v(y)\int_0^\infty \mathrm{e}^{-xy^\lambda}g(x)\mathrm{d}x.\end{equation*} Therefore, $\{\mathcal{L}^\ast g_k\}_{k\in\mathbb{Z}}$ strongly converges in $L^{p'}:$
\begin{equation}\label{lim2'}\lim_{k\to\infty}\|\mathcal{L}^\ast g_k\|_{p'}=0.\end{equation} We obtain
\begin{eqnarray*}\int_0^\infty v^{p'}(y)\biggl(\int_0^\infty \mathrm{e}^{-xy^\lambda}g_k(x)\mathrm{d}x\biggr)^{p'}\mathrm{d}y \ge V_0(z_k)
\biggl(\int_0^{z_k^{-\lambda}}\mathrm{e}^{-xz_k^\lambda}\mathrm{d} x\biggr)^{p'}z_k^{\lambda p'/q'}\\\ge\mathrm{e}^{-p'}V_0(z_k)
\biggl(\int_0^{z_k^{-\lambda}}\mathrm{d}x\biggr)^{p'}z_k^{\lambda p'/q'}
=\mathrm{e}^{-p'}z_k^{-\lambda p'/q}V_0(z_k)
=\mathrm{e}^{-p'}A_\mathcal{L}^{p'}(z_k).\end{eqnarray*} Together with \eqref{lim2'} this implies $\lim\limits_{k\to\infty}A_\mathcal{L}(z_k)=0,$ and now the condition \eqref{comp}(ii) is fulfilled by the arbitrariness of $\{z_k\}_{k\in\mathbb{Z}}.$

{\it Necessity} in (ii), (iii) and (iv) follows by Theorem \ref{T0} from the hypothesis of compactness and, therefore, boundedness of $\mathcal{L}.$

(ii) {\it Sufficiency} of the conditions $B_\mathcal{L}<\infty$ (if $1<q<p<\infty$) and $B_p<\infty$ (if $q=1$) for the compactness of $\mathcal{L}$ is provided by Lemma \ref{L2} and Theorem \ref{T1}(b). Namely, if $1<q<p<\infty$ then Lemma \ref{L2} yields $L^p-L^q$-compactness of $\mathcal{L}_0$ (see \eqref{1}), while norms $\mathcal{L}_1$ and $\mathcal{L}_2$ are estimated by Theorem \ref{T1}(b) as follows: \begin{eqnarray}\|\mathcal{L}_1\|_{L^p\to L^q}\le\beta_2\max\Bigl\{2,\bigl(r/p'\bigr)^{1/r}\Bigr\}\left(\int_{\mathbb{R}^+} \chi_{[0,a]}(t)B_\mathcal{L}(t)\mathrm{d}t\right)^{1/r},\nonumber\\\label{hvostb} \|\mathcal{L}_2\|_{L^p\to L^q}\le\beta_2\left(\int_{\mathbb{R}^+} \chi_{[b,\infty)}(t)B_\mathcal{L}(t)\mathrm{d}t\right)^{1/r}.
\end{eqnarray}  Thus, the condition $B_\mathcal{L}<\infty$ and the estimate \eqref{limit} implies compactness of $\mathcal{L}$ as $a\to 0,$ $b\to\infty.$ If $q=1$ then $\mathcal{L}$ is compact by Lemma \ref{L2}.

{\it Sufficiency} in (iii) and (iv) can be established as follows. Let $B_\mathcal{L}<\infty$ if $0<q<1<p<\infty$ or $B_{q'}<\infty$ if $0<q<1=p.$ By Theorem \ref{T1}(c) we obtain the estimate \eqref{hvostb} for the case $0<q<1<p<\infty$ (with $\beta_3$ instead of $\beta_2$). By the part (iv) of the same theorem we have for $0<q<1=p$\,:
\begin{eqnarray*}\|\mathcal{L}_2\|_{L^p\to L^q}\le\beta_4\left(\int_{\mathbb{R}^+}\chi_{[b,\infty)}(t)B_{q'}(t)\mathrm{d}t\right)^{1/r}.
\end{eqnarray*} Thus, the condition $B_\mathcal{L}<\infty$ (or $B_{q'}<\infty$) yields $\|\mathcal{L}_2\|_{L^p\to L^q}\to 0$ as $b\to\infty.$

Now consider the operator $\mathcal{L}_bf:=\mathcal{L}_0f+\mathcal{L}_1f=L(f\chi_{[0,b]}).$ The hypothesis $B_\mathcal{L}<\infty$ (or $B_{q'}<\infty$) suffices for the boundedness of $\mathcal{L}$ (see Theorem \ref{T0}). Therefore, $\mathcal{L}_b$ is bounded as well. To prove the compactness of $\mathcal{L}_b$ we shall use an extension of Theorem \ref{T2} for the case when an operator $K$ is acting to $L^q$ on the whole $\mathbb{R}^+.$ Similar to \cite[Theorem 5.8]{KZPS}
we consider first a set ${\mathcal M}_h:=\{f\in L^p(a,b)\colon |f|\le h\},$ where $h$ is an arbitrary positive number. Under this condition and in view of ${\rm mes}[0,b]<\infty$ the operator $\mathcal{L}_b$ is bounded from $L^\infty[0,b]$ to $L^q.$ Compactness of $\mathcal{L}_b:{\mathcal M}_h\to L^q$ can be proved similar to \cite[Theorem 5.2]{KZPS}. It remains to note that the rest transformation $\mathcal{L}_b:\{L^p[0,b]\setminus{\mathcal M}_h\}\to L^q$ has a norm tending to $0$ as $h\to+\infty$ (see \cite[Theorem 5.8]{KZPS} for details). Therefore, the operator $\mathcal{L}_b:L^p[0,b]\to L^q,$ $0<q<1\le p<\infty,$ is compact as a limit of compact operators as $h\to +\infty$.

Summing up we can claim that \eqref{L} is compact from $L^p$ to $L^q,$ $0<q<1\le p<\infty,$ on the strength of $\|\mathcal{L}_2\|_{L^p\to L^q}\to 0,$ $b\to\infty,$ compactness of $\mathcal{L}_b$ and
\begin{eqnarray}\label{limit1}
\|\mathcal{L}-\mathcal{L}_b\|_{L^p\to L^q}=\|\mathcal{L}_2\|_{L^p\to L^q}.
\end{eqnarray}

(v) {\it Sufficiency.} Suppose $\esup_{t\in\mathbb{R}^+}{\bar B}_q(t)<\infty$ and \begin{equation}\label{muk}{\rm(i)}\ \lim_{t\to 0}{\bar B}_q(t)=0,{\rm(ii)}\lim_{t\to \infty}{\bar B}_q(t)=0.\end{equation} Put \begin{equation*}\mathcal{L}_af(x):=\mathrm{e}^{-xa^\lambda}\int_a^b f(y)v(y)\mathrm{d}y,\hspace{1cm}x\in\mathbb{R}^+,\end{equation*} where $0<a<b<\infty,$ and note that $\mathcal{L}_a$ is the operator of rank 1 with the norm $$\|\mathcal{L}_a\|_{L^1\to L^q}=q^{-1/q}a^{-\lambda/q}{\bar v}_a(b)<\infty.$$ Besides, $\mathcal{L}_a$ is a majorant for the operator $\mathcal{L}_0,$ which is $L^1-L^q$--bounded with the norm estimated as follows: $$\|\mathcal{L}_0\|_{L^1\to L^q}=\|\mathcal{L}\|_{L^1(a,b)\to L^q}=q^{-1/q}\esup_{a<t<b}{\bar B}_q(t)=:q^{-1/q}M<\infty.$$ Suppose $\{f_n\}_{n\in\mathbb{Z}}$ is an arbitrary bounded sequence in $L^1(a,b)$ and assume $\{f_{n_k}\}$ is its Cauchy subsequence, that is for any $\varepsilon_0>0$ there exists $N(\varepsilon_0)$ such that $$\|f_{n_k}-f_{m_k}\|_{1,(a,b)}<\varepsilon_0, \hspace{1cm}n_k,m_k>N(\varepsilon_0).$$ Put $E_{n_k,m_k}(\varepsilon):=\Bigl\{x\in\mathbb{R}^+\colon \bigl| \mathcal{L}_0f_{n_k}(x)-\mathcal{L}_0f_{m_k}(x)\bigr|>\varepsilon\Bigr\}.$ We have for any $\varepsilon>0:$ \begin{eqnarray*} \int_{E_{n_k,m_k}(\varepsilon)}\mathrm{d}x\le \varepsilon^{-1}\int_{\mathbb{R}^+} \Bigl|\int_a^b\mathrm{e}^{-xy^\lambda}[f_{n_k}(y)-f_{m_k}(y)]v(y)\mathrm{d}y \Bigr|\mathrm{d}x\\
\le \varepsilon^{-1}\int_a^by^{-\lambda}\bigl|f_{n_k}(y)-f_{m_k}(y)\bigr|v(y)\mathrm{d}y\\
\le \varepsilon^{-1}Ma^{-\lambda/q'}\|f_{n_k}-f_{m_k}\|_{1,(a,b)}.  \end{eqnarray*} If $\varepsilon_0=\varepsilon\delta a^{\lambda/q'}M^{-1} $ then $\mu_{L^q}(E_{n_k,m_k}(\varepsilon))<\delta$ as $n_k,m_k>N(\varepsilon,\delta)$ for any $\varepsilon>0,$ $\delta>0.$ Therefore, $\mathcal{L}_0$ is compact in measure. Thus, $\mathcal{L}_0$ is $L^1-L^q$--compact as a transformation majorated by the compact operator $\mathcal{L}_a$ (see \cite[Ch. 2, \S 5.6]{KZPS}).

Further, \begin{eqnarray*}\|\mathcal{L}_1\|_{L^1\to L^q}=q^{-1/q}\esup_{t\in[0,a]}B_q(t)\le q^{-1/q}\sup_{t\in[0,a]}{\bar B}_q(t),\\\|\mathcal{L}_2\|_{L^1\to L^q}\le q^{-1/q}\esup_{t\in[b,\infty)}t^{-\lambda/q}{\bar v}_b(t)\le q^{-1/q}\sup_{t\in[b,\infty)}{\bar B}_q(t).\end{eqnarray*} Since the conditions \eqref{muk} are fulfilled we can state that $\mathcal{L}_1$ and $\mathcal{L}_2$ are operators with small norms, when $a\to 0,$ $b\to\infty.$ Together with compactness of $\mathcal{L}_0$ this implies the compactness of $\mathcal{L}$ from $L^1$ to $L^q$ for all $1\le q<\infty.$

{\it Necessity}. Suppose $\mathcal{L}$ is $L^1-L^q$--compact. Then the claim  $\displaystyle\esup_{t\in\mathbb{R}^+}{\bar B}_q(t)$ $<\infty$ follows from theorems \ref{T0}(iv), \ref{p1} and \ref{p=1}. As for necessity of \eqref{muk}(i), note that
\begin{equation*}\mathcal{L}f=\mathcal{L}(f\chi_{[0,x^{-1/\lambda}]}) +\mathcal{L}(f\chi_{[x^{-1/\lambda},\infty)}):=\mathcal{L}_xf+\mathcal{L}^xf \end{equation*} and $\mathcal{L}_x,$ $\mathcal{L}^x$ are compact. Besides, the condition \eqref{muk}(i) is equivalent to
\begin{equation}\label{muk'}\lim_{k\to -\infty}2^{-\lambda k/q}{\bar v}_0(2^k)=0.\end{equation} Now suppose the contrary. Then, similar to \cite[p. 84]{EGP}, given $\gamma\in(0,1)$ there is a sequence $k_j\to -\infty,$ some $\varepsilon>0$ and functions $f_{k_j}\ge 0,$ $\|f_{k_j}\|_{L^1}\le 1,$ such that \begin{equation*}\int_{0}^{2^{k_j}}f_m(y)v(y)\mathrm{d}y\ge \gamma\,{\bar v}_0(2^{k_j}),\hspace{5mm}\textrm{and}\hspace{5mm} 2^{-k_j\lambda/q}{\bar v}_0(2^{k_j})\ge\varepsilon.\end{equation*}
By continuity of the integral, there are $\beta_{k_j}\in(0,2^{k_j})$ such that \begin{equation*}\int_{\beta_{k_j}}^{2^{k_j}}f_{k_j}(y)v(y)\mathrm{d}y\ge \gamma^2\,{\bar v}_0(2^{k_j}).\end{equation*}
Set $F_{k_j}=f_{k_j}\chi_{(\beta_{k_j},2^{k_j})}.$ Then we have for $k_i$ and $k_j$ such that $2^{k_i+1}<\beta_{k_j}:$ \begin{eqnarray*}\|\mathcal{L}_xF_{k_i}-\mathcal{L}_xF_{k_j}\|_{q}^q=\int_{\mathbb{R}^+} \left|\int_0^{x^{-1/\lambda}} \mathrm{e}^{-xy^\lambda}[F_{k_i}(y)-F_{k_j}(y)]v(y)\mathrm{d}y\right|^q\mathrm{d}x\\
=\lambda\int_{\mathbb{R}^+}s^{-\lambda-1} \left|\int_0^{s}\mathrm{e}^{-(y/s)^\lambda}[F_{k_i}(y)-F_{k_j}(y)]v(y) \mathrm{d}y\right|^q\mathrm{d}s\\=:\|\widetilde{\mathcal{L}_x}F_{k_i}-\widetilde{\mathcal{L}_x}F_{k_j}\|_{q}^q
\ge \|\chi_{(2^{k_i},2^{k_i+1})}(\widetilde{\mathcal{L}_x} F_{k_i}-\widetilde{\mathcal{L}_x}F_{k_j})\|_{q}^q\\
=\|\chi_{(2^{k_i},2^{k_i+1})}\widetilde{\mathcal{L}_x}F_{k_i}\|_{q}^q\\=\lambda\int_{2^{k_i}}^ {2^{k_i+1}}s^{-\lambda-1}\left(\int_{\beta_{k_i}}^{2^{k_i}} \mathrm{e}^{-(y/s)^\lambda}f_{k_i}(y) v(y)\mathrm{d}y\right)^q\mathrm{d}s
\\\ge\mathrm{e}^{-1}\lambda \int_{2^{k_i}}^{2^{k_i+1}}s^{-\lambda-1}\mathrm{d}s\left( \int_{\beta_{k_i}}^{2^{k_i}}f_{k_i}(y)v(y)\mathrm{d}y\right)^q\\ \ge \gamma^{2q}\frac{2^\lambda-1}{2^{\lambda}\mathrm{e}}\,2^{-\lambda k_i}[{\bar v}_0(2^{k_i})]^q\ge \gamma^{2q}\frac{2^\lambda-1}{2^{\lambda}\mathrm{e}}\varepsilon^q>0,
\end{eqnarray*} and $\mathcal{L}_x$ is not compact.

Necessity of \eqref{muk}(ii) can be established by the similar way obtaining a contradiction with the compactness of $\mathcal{L}^x.$ Another way to prove \eqref{muk}(ii) for $q>1$ is analogous to the proof of necessity \eqref{comp}(ii) in the part (i) of this theorem.
\end{proof}
\begin{remark}
The condition $\esup_{t\in\mathbb{R}^+}{\bar B}_q(t)<\infty$ in (v) may be replaced by $\esup_{t\in\mathbb{R}^+}B_q(t)<\infty.$
\end{remark}

\section{COMPACTNESS OF THE STIELTJES TRANSFORM}

We start from boundedness and compactness criteria for the Hardy operator \begin{equation*}H_{\langle c_1,c_2\rangle}f(x):=\psi(x)\int_{c_1}^x f(y)\phi(y)\mathrm{d}y,\hspace{1cm}0\le c_1\le x\le c_2\le\infty,\end{equation*}
with weight functions $\phi^{p'}\in L_{loc}(\mathbb{R}^+)$ and $\psi^{q}\in L_{loc}(\mathbb{R}^+).$ Denote \begin{equation*}\Phi_{c_1}(c_2):=\int_{c_1}^{c_2} \phi^{p'}(y)\mathrm{d}y,\hspace{1cm}\Psi_{c_1}(c_2):=\int_{c_1}^{c_2} \psi^{q}(x)\mathrm{d}x.\end{equation*}
\begin{theorem}\label{H0} {\rm(i)} \cite{Br, EGP, Maz, Muk, Rie}
Let $1<p\le q<\infty.$ The operator $H_{\langle c_1,c_2\rangle}$ is bounded from $L^p\langle c_1,c_2\rangle$ to $L^q\langle c_1,c_2\rangle$ if and only if \begin{equation*}A_{\langle c_1,c_2\rangle}:=\sup_{c_1<t< c_2}A_{\langle c_1,c_2\rangle}(t):=\sup_{c_1<t<c_2}\bigl[\Phi_{c_1}(t)\bigr]^{1/p'}
\bigl[\Psi_t(c_2)\bigr]^{1/q}<\infty,\end{equation*} where $\|H_{\langle c_1,c_2\rangle}\|_{L^p\langle c_1,c_2\rangle\to L^q\langle c_1,c_2\rangle}\approx A_{\langle c_1,c_2\rangle}.$ Besides, $H_{\langle c_1,c_2\rangle}$ is compact if and only if $A_{\langle c_1,c_2\rangle}<\infty$ and \begin{equation*}\lim_{t\to c_1}A_{\langle c_1,c_2\rangle}(t)=\lim_{t\to c_2}A_{\langle c_1,c_2\rangle}(t)=0. \end{equation*}
{\rm(ii)} \cite{A, Maz} If $0<q<1<p<\infty$ or $1<q<p<\infty$ then $H_{\langle c_1,c_2\rangle}$ is bounded from $L^p\langle c_1,c_2\rangle$ to $L^q\langle c_1,c_2\rangle$ if and only if \begin{equation*}B_{\langle c_1,c_2\rangle}:=\left(\int_{c_1}^{c_2}\bigl[\Phi_{c_1}(t)\bigr]^{r/p'} \bigl[\Psi_t(c_2)\bigr]^{r/p} \psi^q(t)\mathrm{d}t\right)^{1/r}<\infty,\end{equation*} where $\|H_{\langle c_1,c_2\rangle}\|_{L^p\langle c_1,c_2\rangle\to L^q\langle c_1,c_2\rangle}\approx B_{\langle c_1,c_2\rangle}.$ Moreover, finiteness of $B_{\langle c_1,c_2\rangle}$ gives a necessary and sufficient condition for the compactness of $H_{\langle c_1,c_2\rangle}$ if $1<q<\infty.$\\[1mm]
{\rm(iii)} \cite{SinStep} Let $0<q<1=p.$ The Hardy operator $H_{\langle c_1,c_2\rangle}$ is bounded from $L^1\langle c_1,c_2\rangle$ to $L^q\langle c_1,c_2\rangle$ if and only if $B_{q<1,\langle c_1,c_2\rangle}<\infty,$ where
\begin{equation*}B_{q<1,\langle c_1,c_2\rangle}:=\left(\int_{c_1}^{c_2}[\esup_{c_1<y<t}\phi(y)]^{q/(1-q)} \bigl[\Psi_t(c_2)\bigr]^{q/(1-q)}
\psi^{q}(t)\mathrm{d}t\right)^{(1-q)/q}\end{equation*} and $\|H_{\langle c_1,c_2\rangle}\|_{L^1\langle c_1,c_2\rangle\to L^q\langle c_1,c_2\rangle}
\approx B_{q<1,\langle c_1,c_2\rangle}.$\\
{\rm(iv)} \cite{EGP} If $p=1\le q<\infty$ then $H_{\langle c_1,c_2\rangle}$ is bounded from $L^1\langle c_1,c_2\rangle$ to $L^q\langle c_1,c_2\rangle$ if and only if $B_{1\le q,\langle c_1,c_2\rangle}<\infty,$ where
\begin{equation*}B_{1\le q,\langle c_1,c_2\rangle}:=\sup_{c_1\le t\le c_2}\bigl[\esup_{c_1<y<t}\phi(y)\bigr]
\bigl[\Psi_t(c_2)\bigr]^{1/q}\end{equation*} and $\|H_{\langle c_1,c_2\rangle}\|_{L^1\langle c_1,c_2\rangle\to L^q\langle c_1,c_2\rangle}
\approx B_{1\le q,\langle c_1,c_2\rangle}.$
\end{theorem}

Some of the results collected in Theorem \ref{H0} are also referenced in \cite{Sin1} and \cite{SinStep}. Besides, similar statement is true for the operator $$H_{\langle c_1,c_2\rangle}^\ast f(x):=\psi(x)\int_x^{c_2} f(y)\phi(y)\mathrm{d}y,\hspace{1cm}0\le c_1\le x\le c_2\le\infty.$$

\begin{theorem}\label{H0ast} {\rm(i)}
If $1<p\le q<\infty$ then $H_{\langle c_1,c_2\rangle}^\ast:L^p\langle c_1,c_2\rangle\to L^q\langle c_1,c_2\rangle$ if and only if \begin{equation*}A_{\langle c_1,c_2\rangle}^\ast:=\sup_{c_1<t<c_2}A_{\langle c_1,c_2\rangle}^\ast(t) :=\sup_{c_1<t<c_2}\bigl[\Phi_t(c_2)\bigr]^{1/p'}
\bigl[\Psi_{c_1}(t)\bigr]^{1/q}<\infty,\end{equation*} with $\|H_{\langle c_1,c_2\rangle}^\ast\|_{L^p\langle c_1,c_2\rangle\to L^q\langle c_1,c_2\rangle}\approx A_{\langle c_1,c_2\rangle}^\ast.$ Moreover, $H_{\langle c_1,c_2\rangle}^\ast$ is compact if and only if $A_{\langle c_1,c_2\rangle}^\ast<\infty$ and $$\lim_{t\to c_1}A_{\langle c_1,c_2\rangle}^\ast(t)=\lim_{t\to c_2}A_{\langle c_1,c_2\rangle}^\ast(t)=0.$$
{\rm(ii)} $H_{\langle c_1,c_2\rangle}$ is bounded from $L^p\langle c_1,c_2\rangle$ to $L^q\langle c_1,c_2\rangle$ for $0<q<1<p<\infty$ or $1<q<p<\infty$ if and only if \begin{equation*}B_{\langle c_1,c_2\rangle}^\ast:=\left(\int_{c_1}^{c_2}\bigl[\Phi_t(c_2)\bigr]^{r/p'} \bigl[\Psi_{c_1}(t)\bigr]^{r/p} \psi^q(t)dt\right)^{1/r}<\infty.\end{equation*} Besides, $\|H_{\langle c_1,c_2\rangle}^\ast\|_{L^p\langle c_1,c_2\rangle\to L^q\langle c_1,c_2\rangle}\approx B_{\langle c_1,c_2\rangle}^\ast$ and the condition $B_{\langle c_1,c_2\rangle}^\ast<\infty$ is necessary and sufficient for the compactness of $H_{\langle c_1,c_2\rangle}^\ast$ if $1\le q<\infty.$\\[1mm]
{\rm(iii)} \cite{SinStep} Let $0<q<1=p.$ The Hardy operator $H^\ast_{\langle c_1,c_2\rangle}$ is bounded from $L^1\langle c_1,c_2\rangle$ to $L^q\langle c_1,c_2\rangle$ if and only if $B^\ast_{q<1,\langle c_1,c_2\rangle}<\infty,$ where
\begin{equation*}B^\ast_{q<1,\langle c_1,c_2\rangle}:=\left(\int_{c_1}^{c_2}[\esup_{t<y<c_2}\phi(y)]^{q/(1-q)} \bigl[\Psi_{c_1}(t)\bigr]^{q/(1-q)}
\psi^{q}(t)dt\right)^{(1-q)/q}\end{equation*} and $\|H^\ast_{\langle c_1,c_2\rangle}\|_{L^1\langle c_1,c_2\rangle\to L^q\langle c_1,c_2\rangle}
\approx B^\ast_{q<1,\langle c_1,c_2\rangle}.$\\
{\rm(iv)} \cite{EGP} If $p=1\le q<\infty$ then $H^\ast_{\langle c_1,c_2\rangle}$ is bounded from $L^1\langle c_1,c_2\rangle$ to $L^q\langle c_1,c_2\rangle$ if and only if $B^\ast_{1\le q,\langle c_1,c_2\rangle}<\infty,$ where
\begin{equation*}B^\ast_{1\le q,\langle c_1,c_2\rangle}:=\sup_{c_1\le t\le c_2}\bigl[\esup_{t<y<c_2}\phi(y)\bigr]
\bigl[\Psi_{c_1}(t)\bigr]^{1/q}\end{equation*}
and $\|H^\ast_{\langle c_1,c_2\rangle}\|_{L^1\langle c_1,c_2\rangle\to L^q\langle c_1,c_2\rangle}
\approx B^\ast_{1\le q,\langle c_1,c_2\rangle}.$
\end{theorem}

Now put $\Lambda:=\biggl(\int_{\mathbb{R}^+}\biggl(\int_{\mathbb{R}^+} \frac{w(x)\mathrm{d}x} {x^\lambda+y^\lambda}\biggr)^{p'}v^{p'}(y)\mathrm{d}y\biggr)^{1/p'},$
\begin{equation*}S_H:=\sup_{t\in\mathbb{R}^+}{\bar v}_0(t)\bigl[\mathcal{W}_t(\infty)\bigr]^{1/q},\hspace{5mm} S_{H^\ast}:=\sup_{t\in\mathbb{R}^+}{\bar v}_t(\infty)t^{-\lambda}\bigl[W_0(t)\bigr]^{1/q},\end{equation*}
\begin{equation*}S_{H,a}(t):={\bar v}_0(t)\bigl[\mathcal{W}_t(a)\bigr]^{1/q},\hspace{5mm} S_{H^\ast,a}(t):={\bar v}_t(a)t^{-\lambda}\bigl[W_0(t)\bigr]^{1/q},\end{equation*}
\begin{equation*}S_{H,b}(t):={\bar v}_b(t)\bigl[\mathcal{W}_t(\infty)\bigr]^{1/q}, \hspace{5mm} S_{H^\ast,b}(t):={\bar v}_t(\infty)t^{-\lambda}\bigl[W_b(t)\bigr]^{1/q}. \end{equation*}

Compactness criteria for the Stieltjes transformation $S:L^p\to L^q$ read
\begin{theorem}\label{St}
{\rm(i)} If $1<p\le q<\infty$ then $S:L^p\to L^q$ is compact if and only if $A_H+A_{H^\ast}<\infty$ and $\lim_{t\to 0}[A_H(t)+A_{H^\ast}(t)]=\lim_{t\to \infty}[A_H(t)+A_{H^\ast}(t)]=0.$\\
{\rm(ii)} Let $0<q<p<\infty$ and $p>1.$ If $q\not=1$ then $S$ is compact if and only if $B_H+B_{H^\ast}<\infty.$ If $q=1$ then $S$ is $L^p-L^1$--compact iff $\Lambda<\infty.$ \\
{\rm(iii)} If $0<q<1=p$ then $S$ is $L^p-L^q$--compact if and only if $B_{1,H}+B_{1,H^\ast}<\infty.$\\
{\rm(iv)} If $p=1\le q<\infty$ then the operator $S:L^p\to L^q$ is compact if and only if $S_H+S_{H^\ast}<\infty$ and  $\lim_{a\to 0}\sup_{0<t<a}[S_{H,a}(t)+S_{H^\ast,a}(t)]=\lim_{b\to \infty}\sup_{b<t<\infty}[S_{H,b}(t)+S_{H^\ast,b}(t)]=0.$
\end{theorem}

\begin{proof} (i) Let $1<p\le q<\infty$ and suppose $A_H+A_{H^\ast}<\infty$ (see \eqref{H1} and \eqref{H2}). Besides, assume
\begin{equation}\label{compH}{\rm(i)}\ \lim_{t\to 0}[A_H(t)+A_{H^\ast}(t)]=0,
{\rm(ii)}\lim_{t\to \infty}[A_H(t)+A_{H^\ast}(t)]=0.\end{equation}
By theorems \ref{H0}(i) and \ref{H0ast}(i) these conditions guaranty the compactness of the operator $H(|f|)+H^\ast\left(|f|\right),$ which is a majorant for the transformation $S$ (see the relation \eqref{rel}). From here the compactness of $S:L^p\to L^q$ ensues by Theorem \ref{TD}.

The condition $A_H+A_{H^\ast}<\infty$ and \eqref{compH} are also necessary for $L^p-L^q$--compactness of $S,$ when $1<p\le q<\infty,$ by standard arguments for the Hardy operators $H$ and $H^\ast.$

(ii), (iii): Let $0<q<p<\infty$ and $p\ge 1.$

If $q\not=1$ we suppose $B_H+B_{H^\ast}<\infty$ for $p>1$ and $B_{1,H}+B_{1,H^\ast}<\infty$ for the case $p=1.$ Compactness of $S$ in the case $p>1$ is guaranteed by $B_{H}+B_{H^\ast}<\infty$ (see theorems \ref{T2} and \ref{TD}). If $p=1$ and $B_{1,H}+B_{1,H^\ast}<\infty$ the compactness of $S$ can be stated similarly to sufficiency of the conditions (iv) in Theorem \ref{lapa}.

If $q=1$ then $S$ is compact by Lemma \ref{L2} provided $\Lambda<\infty.$

Necessity of the condition $B_H+B_{H^\ast}<\infty$ in (ii) and $B_{1,H}+B_{1,H^\ast}<\infty$ in (iii) ensues from the compactness and, therefore, boundedness of the operator $S.$

(iv) It remains to consider the case $p=1\le q<\infty.$ Suppose   $S_H+S_{H^\ast}<\infty,$
\begin{eqnarray}\label{112}{\rm(i)}\ \lim_{a\to 0}\sup_{0<t<a}[S_{H,a}(t)+S_{H^\ast,a}(t)]=0,\\{\rm(ii)}\lim_{b\to \infty}\sup_{b<t<\infty}[S_{H,b}(t)+S_{H^\ast,b}(t)]=0,\nonumber\end{eqnarray} and prove sufficiency of these assumptions for the $L^1-L^q$--compactness of $S.$

In view of \eqref{112} given $\varepsilon>0$ there exist $0<r< R<\infty$ such that \begin{eqnarray}\label{44} \sup_{0<t<r}S_{H,r}<\varepsilon/7,\hspace{1cm}\sup_{0<t<r}S_{H^\ast,r}<\varepsilon/7, \\\label{44'}\sup_{R<t<\infty}S_{H,R}<\varepsilon/7,\hspace{1cm} \sup_{R<t<\infty}S_{H^\ast,R}<\varepsilon/7.\end{eqnarray} Now we divide $S$ into a sum
\begin{equation*}Sf=S_{r,R}f+\sum_{i=1}^2 [S_{r,i}f+S_{R,i}f]\end{equation*} of compact operators \begin{equation*}S_{r,R}f:=\chi_{(r,R)}S(f\chi_{(r,R)}),\hspace{1cm}
S_{r,1}f:=\chi_{[0,R)}S(f\chi_{([0,r]}),\end{equation*}
\begin{equation*}S_{R,1}f:=\chi_{[R,\infty)}S(f\chi_{[0,R)}),\hspace{1cm}
S_{r,2}f:=\chi_{[0,r]}S(f\chi_{(r,\infty)}),\end{equation*}
\begin{equation*}S_{R,2}f:=\chi_{(r,\infty)}S(f\chi_{[R,\infty)}).\end{equation*}
To confirm the compactness of these operators  we shall use a combination of Theorem \ref{Ada} and \cite[Corollary 5.1]{Ev}. That is we need to show that for a given $\varepsilon>0$ there exist $\delta>0$ and points $0<s<t<\infty$ such that for almost all $y\in\mathbb{R}^+$ and for every $h>0$ with $h<\delta$
\begin{eqnarray}\label{rk1}{\rm (i)}\ J_{s}^q(y):=\int_0^s\left|{\bf k}_S(x,y)\right|^q\mathrm{d}x<\varepsilon^q,\\
{\rm (ii)}\ J_{t}^q(y):=\int_t^\infty\left|{\bf k}_S(x,y)\right|^q\mathrm{d}x<\varepsilon^q\nonumber\end{eqnarray} and
\begin{equation}\label{rk}J_{h}^q(y):=\int_{\mathbb{R}^+}\left|{\bf k}_S(x+h,y)-{\bf k}_S(x,y)\right|^q\mathrm{d}x<\varepsilon^q,\end{equation} where ${\bf k}_S(x,y):=w(x) k_S(x,y)v(y).$

We start from the operator $\mathcal{S}:=S_{r,1}+S_{r,R}+S_{R,2}.$ Suppose $h<\delta(\varepsilon)$ and write
\begin{equation*}J_{h,\mathcal{S}}(y)=v(y) \left(\int_{\mathbb{R}^+}w^q(x)\left[\frac{1}{x^\lambda+y^\lambda}- \frac{1}{(x+h)^\lambda+y^\lambda}\right]^q\mathrm{d}x\right)^{1/q}.\end{equation*} For simplicity consider the case $\lambda=1$ and denote \begin{equation*}I_{(c_1,c_2)}^q(y,h):=\int_0^r\frac{w^q(x)\mathrm{d}x} {(x+y)^{q}(x+y+h)^{q}}.\end{equation*} We have \begin{eqnarray*}J_{h,\mathcal{S}}(y)=  h\chi_{[0,r)}(y)v(y)I_{(0,r)}(y,h)
+h\chi_{[0,r)}(y)v(y)I_{(r,R)}(y,h)\\
+h \chi_{(r,R)}(y)v(y)I_{(r,R)}(y,h)
+h \chi_{[R,\infty)}(y)v(y)I_{(r,R)}(y,h)\\
+h \chi_{[R,\infty)}(y)v(y)I_{(R,\infty)}(y,h)=:\sum_{i=1}^5J_{h,i}(y). \end{eqnarray*}
The conditions \eqref{44} and \eqref{44'} imply $J_{h,1}(y)\le 2\varepsilon/7,$ and $J_{h,5}(y)\le 2\varepsilon/7.$
To estimate $J_{h,2}(y)$ note that \begin{equation*}J_{h,2}(y)\le h r^{-1}{\bar v}_0(r)\bigl[\mathcal{W}_r(\infty)\bigr]^{1/q}\le h r^{-1}S_H.\end{equation*}
From here, with $\delta=\varepsilon r/7S_H$ we obtain $J_{h,2}\le\varepsilon/7.$ Analogously, $J_{h,4}\le\varepsilon/7$ if $\delta=\varepsilon r/7S_{H^\ast}.$

For $J_{h,3}$ note that ${\bar v}_r(R)\bigl[W_r(R)\bigr]^{1/q}<M<\infty$ provided  $w\in L^q_{loc}(\mathbb{R}^+)$ and $v\in L^\infty_{loc}(\mathbb{R}^+)$. Therefore, $J_{h,3}(y)\le hM r^{-2}$ and $J_{h,3}(y)\le \varepsilon/7$ if $\delta=\varepsilon r^2/7M.$

Summing up, we obtain $J_{h,\mathcal{S}}(y)<\varepsilon$ for almost all $y\in\mathbb{R}^+,$ that is the condition \eqref{rk} is satisfied. Fulfillment of the claims \eqref{rk1} ensues from \eqref{44} and \eqref{44'} with $s=r$ and $t=R.$ Thus, the sum $S_{r,1}+S_{r,R}+S_{R,2}$ is compact.

Compactness of the operator $S_{r,2}$ can be demonstrated as follows. The condition \eqref{rk1}(ii) is automatically fulfilled with $t=r.$ To demonstrate \eqref{rk1}(i) note that $\|S_{r,2}\|_{L^1\to L^q}\approx{\bar v}_r(\infty)r^{-\lambda}\bigl[W_0(r)\bigr]^{1/q}\le S_{H^\ast}<\infty.$
Hence, given $\varepsilon>0$ there exists $0<s\le r$ such that $J_{s,S_{r,2}}(y)<\varepsilon.$ The condition \eqref{rk} may be shown analogously with $\delta=\varepsilon r^\lambda/S_{H^\ast}.$ Similar arguments work for the operator $S_{R,1}.$

Necessity of the conditions $[S_H+S_{H^\ast}]<\infty$ and \eqref{112} follow from \cite[Lemma 1, Theorem 1]{EGP} and the relation \eqref{rel}.
\end{proof}
\begin{remark} In some special cases the compactness of $S$ can be established through the Laplace transformation \eqref{L}. Indeed, by the representation
\begin{equation}\label{rep}S=\mathcal{L}_w^\ast\mathcal{L}_v
\end{equation} with $\mathcal{L}_v\equiv\mathcal{L}$ and $\mathcal{L}_w^\ast f(x):=w(y)\int_{\mathbb{R}^+}\mathrm{e}^{-xy^\lambda}f(x)\mathrm{d}x$
we are able to state the compactness of $S:L^p\to L^q$ if the conditions of Theorem \ref{lapa} are fulfilled for either $\mathcal{L}_v:L^p\to L^r$ or $\mathcal{L}_w:L^{q'}\to L^{r'}$ of the form
$\mathcal{L}_wf(x):=\int_{\mathbb{R}^+}\mathrm{e}^{-xy^\lambda}f(y) w(y)\mathrm{d}y.$ Here the parameter $r'$ is such that $r'=r/(r-1)$ for $r>1.$ In particular, if $w\equiv v$ and $p=q'\le q=p'$ then $S:L^p\to L^{p'}$ is compact if \begin{equation*}A=\sup_{t\in\mathbb{R}^+}A(t):= \sup_{t\in\mathbb{R}^+}t^{-\lambda/2}\bigl[V_0(t)\bigr]^{1/p'}<\infty\end{equation*} and $\lim_{t\to 0}A(t)=\lim_{t\to \infty}A(t)=0.$ Necessity of these conditions can be proved by Theorem \ref{T3} and the inequality $A_S(t)\ge A(t)^2.$\end{remark}

\section{CASES $p=\infty$ AND $q=\infty.$}
In view of the representation \eqref{rep} compactness criteria for the operator $\mathcal{L}$ from $p=\infty$ and/or to $q=\infty$ may be useful to state compactness of the transformation $S.$ To obtain the criteria we traditionally start from boundedness.

Let $\mathcal{L}$ be the transform from $L^\infty$ to $L^\infty.$ Then $\mathcal{L}$ is bounded if and only if \begin{equation}\label{c1} C_1:=\|v\|_1<\infty.\end{equation} If $\mathcal{L}$ is acting from $L^\infty$ to $L^q,$ $1<q<\infty,$ we remind that by Lemma \ref{L0} \begin{eqnarray*}
\int_{\mathbb{R}^+}\biggl(\int_{\mathbb{R}^+}\mathrm{e}^{-xy^\lambda}f(y)v(y) \mathrm{d}y\biggr)^q\mathrm{d}x
\approx\int_{\mathbb{R}^+}\biggl(\int_{0}^yf(t)v(t)\mathrm{d}t\biggr)^q
y^{-\lambda-1}\mathrm{d}y.\end{eqnarray*} Together with \cite[\S 1.3.2, Th.1]{Maz} this gives \begin{eqnarray*}\|\mathcal{L}\|_{L^\infty\to L^q}\approx \biggl(\int_{\mathbb{R}^+}t^{-\lambda}\biggl(\int_{0}^tv(y)\mathrm{d}y\biggr)^{q-1} v(t)\mathrm{d}t\biggr)^{1/q}=:C_{q>1}.\end{eqnarray*} If $q=1$ then \begin{equation*}\|\mathcal{L}\|_{L^\infty\to L^1}=\int_{\mathbb{R}^+}y^{-\lambda}v(y)\mathrm{d}y=:C_{q=1}.\end{equation*}

Let $q<1.$ Then $\mathcal{L}$ is bounded from $L^\infty$ to $L^q$ if $C_{q>1}<\infty.$ If $\mathcal{L}:L^\infty\to L^q$ boundedly then \begin{equation*}C_{q<1}:=\int_{\mathbb{R}^+}t^{-\lambda/q}v(t)\mathrm{d}t<\infty. \end{equation*}

Suppose $\mathcal{L}$ is an operator from $L^p$ to $L^\infty,$ when $p\ge 1.$ If $p=1$ then  \begin{equation*}\|\mathcal{L}\|_{L^1\to L^\infty} =\|v\|_{\infty}=:C_\infty.\end{equation*} If $p>1$ then $\mathcal{L}^\ast$ is acting from $L^1$ to $L^{p'},$ $p'>1.$ This yields \begin{equation*}\|\mathcal{L}\|_{L^p\to L^\infty}=\|\mathcal{L}^\ast\|_{L^1\to L^{p'}}=\|v\|_{p'}=:C_{p'}.\end{equation*} The estimate $\|\mathcal{L}\|_{L^p\to L^\infty}\le C_{p'}$ here and in \eqref{c1} ensues from Theorem \ref{p=1}. The reverse inequality follows by applying a function $g(x)=\chi_{[0,t^{-\lambda})}(x)t^\lambda$ into the $L^1-L^{p'}$--norm inequality for the operator $\mathcal{L}^\ast.$

Let $C_q$ denote $C_{q>1}$ or $C_{q=1}$ depending on a range of the parameter $q.$ Now we can state the required compactness criteria for $\mathcal{L}.$
\begin{theorem}
{\rm(i)} Let $1\le q<\infty.$ $\mathcal{L}$ is compact from $L^\infty$ to $L^q$ iff $C_q<\infty.$\\
{\rm(ii)} Let $q<1.$ The operator $\mathcal{L}$ is $L^\infty-L^q$--compact if $C_{q>1}<\infty.$ If $\mathcal{L}$ is compact from $L^\infty$ to $L^q$ then $C_{q<1}<\infty.$\\
{\rm(iii)} If $1<p\le\infty$ then $\mathcal{L}$ is $L^p-L^\infty$--compact if and only if $C_{p'}<\infty.$
\end{theorem}
\begin{proof} Proof of the results rests on \cite[Theorem 5.2]{KZPS} and \cite{Ad}.
\end{proof}
\begin{remark}
The operator $\mathcal{L}$ can not be compact from $L^1$ to $L^\infty$ for any $v.$
\end{remark}
Since the Stieltjes transformation $S$ is two-weighted then its compactness criteria for $p=\infty$ and/or $q=\infty$ can be derived from the results of Sec. 4.

{\bf Acknowledgements.} The European Commission is grant-giving authority for the research of the author (Project IEF-2009-252784). The work is also partially supported by Russian Foundation for Basic Research  (Project No. 09-01-98516) and Far-Eastern Branch of the Russian Academy of Sciences (Project No. 09-II-CO-01-003).



%% References


\bibliographystyle{model1-num-names}
\bibliography{<your-bib-database>}






% \end{thebibliography}

\end{document}